\documentclass[12pt,a4paper,reqno]{amsart}
\usepackage[vmargin=2cm, hmargin=2cm]{geometry}
\usepackage{amsfonts,amsthm,amsmath,amssymb,amscd}
\usepackage{bbm}
\usepackage[usenames,dvipsnames]{color}
\usepackage{pstricks}

\usepackage{enumerate}
\theoremstyle{definition}
\newtheorem{dfn}{Definition}
\newtheorem{defin}[dfn]{Definition}
\newtheorem{rem}[dfn]{Remark}
\newtheorem*{xrem}{Remark}

\theoremstyle{plain}

\newtheorem{thm}[dfn]{Theorem}

\newtheorem{prop}[dfn]{Proposition}

\newcommand{\N}{{\mathbb N}}
\newcommand{\1}{{\mathbbm 1}}
\newcommand{\esssup}{\operatorname{esssup}}
\DeclareMathOperator{\diam}{diam}

\numberwithin{equation}{section}

\title[Hausdorff measure by recurrence]{Estimating the Hausdorff measure using recurrence}
\author{\L ukasz Pawelec}
\begin{document}
\date{}
\begin{abstract}
We show a new method of estimating the Hausdorff measure of a set from below. The method requires computing the subsequent closest return times of a point to itself.
\end{abstract}
\address{{\L}ukasz Pawelec, Department of Mathematics and Mathematical Economics, SGH Warsaw School of Economics, 
al.~Niepodleg\l{}o\'{s}ci~162, 02-554 Warszawa, Poland}
\email{LPawel@sgh.waw.pl}
\maketitle

\section{Introduction}
Let $(X,d)$ be a separable metric space and $(T,\mu)$ a transformation preserving a Borel, probability measure. The classical Poincar\'{e} lemma in such a setting gives that \[\displaystyle\liminf_{n\to\infty}d(x,T^n(x))=0 \mbox{\; for $\mu$--almost every $x$}.\] The historically  first attempt at strengthening this result came in a paper by M.~Boshernitzan \cite{Bosh}, who proved that $d(x,T^n(x))\approx n^{-1/\beta}$, where $\beta$ is the Hausdorff dimension of $X$. Precisely speaking, he gave two results, which we state now.
 
For a dynamical system $(X,T)$ preserving a probability Borel measure $\mu$.
\begin{align}\label{eq:wstepbosh1}
&\mbox{If $H_\beta(X) <+\infty$, then \;} \liminf_{n\to \infty} \;n^{1/\beta}d(T^n(x),x) < +\infty \mbox{\quad for $\mu$--a.e. $x$.}\\
\label{eq:wstepbosh2}
&\mbox{If $H_\beta(X) =0$, then \;} \liminf_{n\to \infty} \;n^{1/\beta}d(T^n(x),x) =0 \mbox{\quad for $\mu$--a.e. $x$.}
\end{align}  

 The second result from that paper states, that if the preserved probability measure $\mu=H_\beta$, then
\begin{equation}\label{eq:wstepbosh3}
\liminf_{n\to \infty} \;n^{1/\beta}d(T^n(x),x) \leq 1 \mbox{\quad for $\mu$--a.e. $x$.}
\end{equation}

%
There has been a lot of development in the area, for an introduction into quantitative recurrence, see e.g. \cite{BS}. 

In this paper we will be interested in showing some new bounds on the recurrence speed. We will prove a generalisation of Boshernitzan's result, but the main new idea is to show how to use this improved result to get an estimate from below of the Hausdorff measure of a fractal set. We discuss this on an easy example. An upcoming paper with M.~Urba\'{n}ski \cite{PU} shows a more interesting application, namely for Cantor sets defined by the so-called Denjoy maps (i.e. we show a bound from below on the Hausdorff measure of the minimal set occurring for a $\mathcal{C}^{1+\alpha}$ diffeomorphism on the circle which is \emph{only} semi-conjugate to a rotation). 

The idea of the method comes from the author's PhD Thesis.

The paper is organised as follows. In the next section we give the needed definitions, state the relevant theorems and sketch the new technique. In Section \ref{sec:ex} we show the method of estimating the Hausdorff measure on an example. Section \ref{sec:impr} is filled with additional comments, improvements and limitations of the method. Finally, Section \ref{sec:proof} is devoted to the proof of Theorem \ref{thm:main}.

\section{Definitions and Theorems}
Throughout this paper we will assume that $(X,d)$ is a metric space and $T\colon~X\to~X$ a Borel measurable map; $\mu$ is a $T$--invariant, ergodic, probability, Borel measure on $X$.

As we are working with subtle measure estimates it seems prudent to put here the precise definitions in use in this paper.

We will use the (most common) version of the definition of the Hausdorff measure.
\begin{defin}The outer Hausdorff measure is the following
\[H_\beta(Y)= \lim_{r\to 0} \inf\left\{\sum_{k=1}^\infty (\diam U_k)^\beta : \forall_k \diam U_k <r \mbox{ and } Y\subset \bigcup_{k=1}^\infty U_k\right\},\]
where the infimum is take over all countable covers of $Y$ satisfying the conditions as stated. By Carath\'eodory's extension this gives the (typical) Hausdorff measure. 
\end{defin}
The next definition is also standard.
\begin{defin}
The Hausdorff dimension of the set $Y$ is given by the formula
\[\dim_H(Y)=\inf\{\beta \geq 0 : H_\beta(Y)=0\}.\]
\end{defin} 

We will now state a new version of Boshernitzan's estimate \eqref{eq:wstepbosh3}. In contrast to his result we do not assume that the preserved measure $\mu=H_\beta$.
\begin{thm}\label{thm:main}
With the assumptions on the dynamical system as above, for any $\alpha>0$ and for $\mu$ -- almost every $x\in X$ we have
\begin{align}
&\liminf_{n\to \infty} n\Big(d(T^n(x),x)\Big)^\alpha\leq g(x):=\limsup\limits_{r\to 0}\frac{H_\beta(B(x,r))}{\mu(B(x,r))}
\label{eq:thmain}
\end{align}
\end{thm}
\begin{xrem}
Note that $g(x)$ may be equal to $0$ or $+\infty$. The statement still holds.
\end{xrem}
The rather simple proof utilises the idea by M.~Boshernitzan and some techniques from ergodic theory. We postpone it till the last section.

\bigskip

This result shows that the behaviour of the recurrence is (may be) governed by the Hausdorff measure of the space. We will try to apply this in a reverse manner: if we could compute/estimate the lower limit of the speed of recurrence, then this would give us some information on the Hausdorff measure.

More precisely, if we can show that the lower limit on the LHS of \eqref{eq:thmain} is  positive for some $\alpha>0$, then we will get the lower bound on the density (and so on the $\alpha$--Hausdorff measure of the space). Also, this would trivially give $\dim_H(X)\geq \alpha$.

Regarding the dimension, note that there is a unique value $\alpha^*\in [0,+\infty]$ such that 
  \begin{align*}
\liminf_{n\to \infty} n\big(d(T^n(x),x)\big)^\alpha = +\infty& \mbox{\quad for all $\alpha<\alpha^*$\quad and}\\
\liminf_{n\to \infty} n\big(d(T^n(x),x)\big)^\alpha = 0& \mbox{\quad for all $\alpha>\alpha^*$.}
\end{align*}
Theorem \ref{thm:main} now gives that $\dim_H(X)\geq \alpha^*$.

\medskip

Note that \emph{a~priori} we may take any map on the space, as long as it preserves some Borel, probability, ergodic measure $\mu$. However, we ought to take a map with \emph{poor} mixing properties because of a result that requires another well-known definition. 

\begin{defin}We say that a dynamical system has an \emph{exponential decay of correlations} in Lipschitz--continuous functions (denoted by $\mathcal{L}$), if 
there exist $\gamma \in (0,1)$ and $C<+\infty$, such that for all $g\in \mathcal{L}$, all $f \in L_1(\mu)$ and every $n\in \N$, we have 
\begin{equation}\label{eq:decaydef}
\left|\mu\left(f\circ T^n\cdot g\right)-\mu(g)\cdot\mu(f)\right| \leq C\gamma^n||g||_\mathcal{L}\mu(|f|),
\end{equation}
where $||\cdot||_\mathcal{L}$ denotes the typical norm of the space of Lipschitz functions.
\end{defin}
The simplified version (stronger assumptions) of Theorem~3.1. from \cite{Paw} states that
\begin{thm}
With the assumptions on the dynamical system as above, if $\mu \approx H_\alpha$ and the system has an exponential decay of correlation in Lipschitz--continuous functions, then
 \begin{equation}\label{eq:thloglog}
\liminf_{n\to \infty} \;\left(n\ln\ln n \right)^{1/\alpha}d(T^n(x),x) =0.
\end{equation}
\end{thm} 
Which is the opposite of what we want (namely a positive lower limit). Thus, for the map to be useful to our method it needs to be slowly mixing. Typical examples of such maps include the irrational rotations on $\mathcal{S}^1$, Feigenbaum maps or the \emph{adding machine} map, which we utilise below.

\section{Example}\label{sec:ex}

Our example will be arguably the simplest of fractal sets -- the one-third Cantor set. We will estimate from below the dimensional density $g(x)$ for all values of $\alpha$. As it turns out, we will get a meaningful result for $\alpha$ equal to the Hausdorff dimension of the Cantor set, leading to a bound on the Hausdorff dimension and the Hausdorff measure, both from below.

As mentioned, we will utilise a so-called \textit{adding machine} map. We recall the definition now.

\medskip

Every point $x$ in the Cantor set $C$ has a unique coding $(x_n)_{n=1}^{\infty}$ using symbols 0 and 1. The first symbol is $0$ if the point is to the left of $1/2$ and $1$ if it is to the right. The second symbol decides if the point is on the left or on the right of the second level segments, etc. The relation between coding and the point on the real line  is  $\displaystyle x=\sum\limits_{n=1}^{+\infty} \frac{2x_n}{3^n}$. It follows that the (Euclidean) distance between points $x$ and $y$ is given by a formula $\displaystyle |x-y|= 2\left|\sum\limits_{n=1}^{+\infty} \frac{x_n-y_n}{3^n}\right|$. 

The map $T$ on the coding space, is defined by an inductive scheme:
\begin{enumerate}[A)]
  \item Start with the first symbol: $n=1$.
	\item If the symbol $x_n= 0$, then add 1 to it (new $(Tx)_n=1$) and finish.
	\item If the symbol $x_n=1$, then make it equal to 0 (new $(Tx)_n=0$), increase $n$ by 1, and return to (B).
\end{enumerate}
In other words --- we scan the code for the first digit of $(x_n)$ equal to $0$, set it to $1$ and set all the previous digits (i.e. $(x_k)$ for $k<n$) to $0$.

Note that this `program' will run indefinitely, if our point $x$ has code $[111\ldots]$ (i.e. if $x=1$), but mathematically this is not an issue as we may set $T(1)=0$.

This map is called an \emph{adding machine}, because it is equivalent to adding 1 to the first digit of a binary number, where the digits are written in reverse order.
This transformation is a piecewise isometry and it preserves the Cantor measure $\mu$ (defined to be equally distributed on the cylinders of the same level/size).  

\begin{figure}[ht]
\centering
{
\begin{pspicture}(0,-1.62)(11.144375,1.62)
\psline[linewidth=0.04cm](0.0,0.0)(1.2,0.0)
\usefont{T1}{ptm}{m}{n}
\rput(7.8478127,-0.29){10..}
\usefont{T1}{ptm}{m}{n}
\rput(0.586875,-0.29){00..}
\usefont{T1}{ptm}{m}{n}
\rput(3.055625,-0.29){01..}
\usefont{T1}{ptm}{m}{n}
\rput(10.731093,-0.29){111..}
\psline[linewidth=0.04cm](7.2,0.0)(8.4,0.0)
\usefont{T1}{ptm}{m}{n}
\rput(9.531094,-0.29){110..}
\psline[linewidth=0.04cm](2.4,0.0)(3.6,0.0)
\psline[linewidth=0.04cm](9.6,0.0)(10.0,0.0)
\psline[linewidth=0.04cm](10.4,0.0)(10.8,0.0)
\psbezier[linewidth=0.04,arrowsize=0.05291667cm 2.0,arrowlength=1.4,arrowinset=0.4]{->}(0.6,0.4)(1.4,1.6)(7.0,1.6)(7.8,0.4)
\psbezier[linewidth=0.04,arrowsize=0.05291667cm 2.0,arrowlength=1.4,arrowinset=0.4]{->}(7.8,-0.6)(7.6,-1.2)(3.2,-1.2)(3.0,-0.6)
\psbezier[linewidth=0.04,arrowsize=0.05291667cm 2.0,arrowlength=1.4,arrowinset=0.4]{->}(3.0,0.4)(3.8,1.6)(9.4,1.6)(10.2,0.4)
\psbezier[linewidth=0.04,arrowsize=0.05291667cm 2.0,arrowlength=1.4,arrowinset=0.4]{->}(9.8,-0.6)(9.8,-1.4)(1.0,-1.4)(1.0,-0.6)
\psbezier[linewidth=0.04,arrowsize=0.05291667cm 2.0,arrowlength=1.4,arrowinset=0.4]{->}(10.6,-0.6)(10.6,-1.6)(0.2,-1.6)(0.2,-0.6)
\end{pspicture} 
}
\caption{\emph{Adding machine} transformation on a Cantor set. The map in the neighbourhood of the point $1111\ldots$ is drawn only up to the cylinder of length 3.}
\end{figure}

Let us start computing the recurrence rate by taking the point $z^0=0=[0000\ldots]$ and denote the forward iterates as $T^n(z^0)=z^n$. 
\[z^1=\frac{2}{3}=[100\ldots],\;z^2=\frac{2}{9}=[010\ldots],\;z^3=\frac{8}{9}=[110\ldots],\;z^4=\frac{2}{27}=[0010\ldots].\]

To calculate the lower limit (LHS) of \eqref{eq:thmain} we only need to look at the subsequent \emph{closest returns}, i.e. we can ignore all $n$ for which there exists $k<n$ such that $|T^k(z)-z|\leq|T^n(z)-z|.$ For our point $z^0$ (and in fact any starting point) it is obvious that those returns will occur for the iterates being powers of~$2$. More precisely,
\begin{align*}\left|T^{2^n}(z^0)-z^0\right|=\frac{2}{3^{n+1}} &\mbox{\hskip1cm for all $n\geq 0$, }\\
\left|T^{k}(z^0)-z^0\right|> \frac{2}{3^{n+1}} &\mbox{\hskip1cm for all $0<k<2^{n}$}.
\end{align*}
Taking any $\alpha>0$ we get the following
\begin{equation}\label{eq:cant1} \liminf_{k\to +\infty} k\left|T^k(z^0)-z^0\right|^\alpha = \lim_{n\to+\infty} 2^n\Big(\frac{2}{3^{n+1}}\Big)^\alpha = \lim_{n\to+\infty} \frac{2^\alpha}{3^\alpha}\left(\frac{2}{3^{\alpha}}\right)^n. \end{equation} 

Obviously, $z^0$ is not a typical point in this system. However, the general calculation is not that different. Take any point $x\in C$ and look at its code -- $[x_1x_2x_3\ldots]$. As before, we only need to look at iterates that are of form $2^n$. The point $T^{2^n}(x)$ will have the first $n$ symbols identical and the $(n+1)$--st symbol will be different. What we do not control/know are the later symbols, which can lower the distance slightly, e.g. the distance between $[100\ldots]$ and $[010\ldots]$ is equal to $4/9$. However, it is easy to write down all the possibilities. 
\begin{align*}\left|T^{2^n}(x)-x\right|=\frac{2}{3^{n+1}} &\mbox{\hskip1cm if $x_{n+1}=0$, }\\
\left|T^{2^n}(x)-x\right|=\frac{4}{3^{n+2}} &\mbox{\hskip1cm if $x_{n+1}=1$ and  $x_{n+2}=0$,}\\
\left|T^{2^n}(x)-x\right|> \frac{2}{3^{n+1}} &\mbox{\hskip1cm if $x_{n+1}=1$ and  $x_{n+2}=1$. }
\end{align*}
To sum up -- the \emph{worst case} is when we add $1$ at the place where there is a symbol $1$ followed by a $0$. 


Repeating \eqref{eq:cant1} for a general point we get a slightly worse estimate
\begin{equation}\label{eq:cant2} \liminf_{k\to +\infty} k\left|T^k(x)-x\right|^\alpha \geq \lim_{n\to +\infty}\Big(\frac{4}{9}\Big)^\alpha\left(\frac{2}{3^{\alpha}}\right)^n=\lim_{n\to +\infty}\Big(\frac{4}{9}\Big)^\alpha\left(3^{\log_32-\alpha}\right)^n.\end{equation}
So if we take any $\alpha<\log_32$, we see that the lower limit is infinite so by using Boshernitzan's result \eqref{eq:wstepbosh1} we know that the Hausdorff measure $H_\alpha(C)$ is infinite, so the Hausdorff dimension $\mathrm{HD}(C)\geq \log_32$. 

Take $\alpha=\log_32$ and the Cantor measure $\mu$. Now Thm.~\ref{thm:main} gives that $g(x)\geq \big(\frac{4}{9}\big)^\alpha$ for all $x$ (where $g(x)=\frac{dH_\alpha}{d\mu}$). So
\begin{equation}\label{wymCant}
H_{\log_32}(C) = \int_C g(x) d\mu(x) \geq \mu(C)\left(\frac{4}{9}\right)^{\log_32}\approx 0.6.	
\end{equation}
This is not a very strong result --- in reality $H_{\log_32}(C)=1$, but on the other hand, the estimate has been acquired with little effort. The next section is dedicated to comments on improving this lower bound.

\medskip

Note that, it is easy to apply this technique to other self-similar sets, which allow symbolic coding, e.g. the Sierpi\'nski triangle. Unfortunately, the \emph{unoptimality} of the lower bound may (and typically will) remain.

On the other hand, the coding is not strictly necessary. If a system has slow recurrence properties, then one could get meaningful results as well. An example of this is in a paper with M.~Urba\'{n}ski \cite{PU}, where the underlying dynamics is that of an irrational rotation on a circle (which for numbers with bad Diophantine properties is in fact slowly recurrent).

\section{Improvements and comments} \label{sec:impr}
\subsection{Changing the metric}
In the calculation above we used the Euclidean metric on the real line. However, on the Cantor set there is another natural metric, coming from the symbolic representation. Define 
\[d(x,y) = d\big((x_n),(y_n)\big) = 3^{1-\min\{k\geq 1\, :\, x_k\neq y_k \}}.\]
Then the diameter of $C$ stays equal to 1. Also, the diameters of all the cylinder sets in this metric is equal to the diameters in the Euclidean one. And the Hausdorff measure (and dimension) are exactly as it was in the Euclidean case.

Let us check what happens to our recurrence estimates if we take this metric.

\noindent For any $z\in C$ we trivially get
\begin{align*}d\big(T^{2^n}(z), z\big)=\frac{1}{3^{n}} &\mbox{\hskip1cm for all $n\geq 0$, }\\
d\big(T^{k}(z), z\big)\geq \frac{1}{3^{n}} &\mbox{\hskip1cm for all $0<k<2^{n}$}.
\end{align*}

Inserting this into the \emph{liminf} estimates yields

\begin{equation} \liminf_{n\to +\infty} k\left|T^k(x)-x\right|^\alpha = \lim_{n\to +\infty}2^n\big(3^{-n}\big)^\alpha=\lim_{n\to +\infty}\big(3^{\log_32-\alpha}\big)^n.\end{equation}
Now, setting $\alpha=\log_32$ we get the estimate on the density $g(x)\geq1$, which in turns gives $H_\alpha(C)\geq1$.

We see that using this metric we get the optimal estimate. 
 


\subsection{Irremovable obstacle}
Let us return to the Euclidean metric. One could ask a very natural question -- would some different map yield a better estimate?

And while it is possible that there exists a map with even slower recurrence, there does not seem to be any chance of improving up to the optimal lower bound. This is shown by a result of Boshernitzan and Delecroix, \cite{BD}, which we will utilise below.

To see the problem, let us try to apply our method to a circle $S$ of length 1. To get the best bound we need to find a map $T$ on the circle (preserving some probability measure $\mu$)  for which
\begin{equation}
\liminf_{n\to \infty} nd(T^n(x),x) \geq 1,
\label{eq:excircle}
\end{equation}
for $\mu$-a.e. $x\in S$. This would prove that $H_1(S)\geq 1$.

First, let us see what should we assume on the measure. Its support needs to be the entire circle (we get nonsense otherwise). Also, the dimension of the measure needs to be 1 (reason as before). Finally, as the circle is geometrically identical at any point so should be the measure -- leaving as only with the Lebesgue measure. \emph{The last argument is not precise at all, but we are not actually proving anything here, so for the sake of clarity let us leave it like that.}

There is still plenty of maps preserving the Lebesgue measure. The simplest of those are the rotations by angle $\gamma$, denoted by $R_\gamma$. Then the recurrence speed does not depend on the starting point, but on the continued fraction expansion of $\gamma$. It is very well studied subject. By the classic result of Khinchin we know that the slowest return speed happens for the rotation by the golden mean (minus one) $\varphi = \frac{\sqrt{5}-1}{2}$. Khinchin's Theorem also states that
\begin{equation}
\liminf_{n\to \infty} nd(R_\gamma^n(x),x) \leq \frac{1}{\sqrt{5}},
\label{eq:Khinch}
\end{equation} 
with the equality for $\varphi$.

This shows that taking only the rotations, we have no chance of realising \eqref{eq:excircle}. And Boshernitzan and Delecroix generalise this to prove (in \cite{BD}) that inequality \eqref{eq:Khinch} is true for \emph{all} maps preserving the Lebesgue measure. This shows that the method shown here as han irremovable obstacle in achieving the best bound. At least on the circle, but their proof suggest this would happen on every space. 

Actually, their proof indicates that there exists a constant \emph{correction term} (depending on the dimension of the space, and perhaps slightly on the geometry of the space) which one could apply to get the correct measure (for the circle this would be $\frac{1}{\sqrt{5}}$). Unfortunately, making this argument precise would require very general results on the optimal packing of points in rather arbitrary sets.
 
\subsection{Dependence on the dimension}

Their proof also suggests that the scale of the \emph{unoptimality} of the lower bound (i.e. the difference between the obtained result and the true Hausdorff measure) depends on the dimension. In fact, it should shrink to zero as the dimension goes to zero. 

We cannot prove this general result here. What we can do, however, is show the phenomenon for basic Cantor sets.

\medskip 

Let us compute the lower bound on the measure of the Cantor sets of varying dimensions.
Fix $0<a<\frac{1}{2}$. The Cantor set $C_a$ in question is given by the maps:
\begin{equation}
	f_0(x)=ax, \qquad f_1(x)=ax+(1-a)=a(x-1)+1.
\end{equation}

The Hausdorff dimension of this set is trivially computed $\dim_H(C_a)=\frac{\log(2)}{-\log(a)}$.

We define the coding as before, take the same adding machine map and repeat the calculation as in the example. We get for any $x\in C_a$ (where $1-2a$ is the size of the gap between intervals of the same order, to which we may add the length of the next-level cylinder $a^2$)
\begin{equation}
\left|T^{2^n}(x)-x\right|\geq a^{n}(1-2a+a^2).
\end{equation}
where the inequality becomes equality for those $n$'s when $x_n=1$ and $x_{n+1}=0$ (exactly as before). Putting this into the lower limit yields
\begin{equation} 
\liminf_{k\to +\infty} k\left|T^k(x)-x\right|^\alpha \geq \lim_{n\to +\infty}2^n\left(a^n(1-2a+a^2)\right)^\alpha=\lim_{n\to +\infty}(1-a)^{2\alpha}\left(2^{1+\alpha\log_2(a)}\right)^n.
\end{equation}
Put $\alpha=\frac{\log(2)}{-\log(a)}$, so that the sequence is constant. Then we see
	\begin{equation}
	H_\alpha(C_a)\geq (1-a)^{2\alpha}.
\end{equation}
This expression goes to one as $a$ goes to zero (note that then so does $\alpha$). Thus the difference between the true Hausdorff measure (equal to one in this case) and our estimate does disappear in the limit.

\section{Proofs} \label{sec:proof}
The proof of Theorem \ref{thm:main} is divided into a few steps. First we prove 
\begin{prop} \label{fingrandens}
With the assumptions on the dynamical system as above, in addition suppose that $H_\alpha \ll \mu$ for some $\alpha>0$, and denote the corresponding density by $g:=\frac{dH_\alpha}{d\mu}$. Then for $\mu$ -- almost every $x\in X$ we have
\begin{equation}
\liminf_{n\to \infty} \;n^{1/\alpha}d(T^n(x),x) \leq (\esssup g)^{1/\alpha}.
\end{equation}
\end{prop}

\begin{xrem} Note that $g$ is the inverse of the usually taken density.\end{xrem}

\begin{proof} First, if $g$ is unbounded, then the inequality is trivial. Denote $\beta:=\frac{1}{\alpha}$ and $s:=\esssup\, g <+\infty$. In this notation we need to show that $\mu(D)=0$, where
\begin{equation}
D = \{x\in X : \liminf_{n\to\infty} \;n^\beta d(T^n(x),x) > s^\beta\}.
\end{equation}
Take any $\varepsilon>0$ and define   
\begin{equation}
D_\varepsilon:=\{x\in X : n^\beta d(T^n(x),x) > \left(\frac{s}{1-\varepsilon}\right)^\beta, \mbox{ for all $n\geq 1$ such that } d(T^n(x),x) < \varepsilon\}.
\end{equation}
It suffices to show that for any $\varepsilon >0$, this set has measure zero. 

Assume the opposite, i.e. $\mu(D_\varepsilon) >0$ for some fixed $\varepsilon$. We will prove that this implies $H_\alpha(D_\varepsilon) >0$. 

Define $\tau(x)$ as the first return map of a point $x$ into $D_\varepsilon$. This map preserves the conditional measure $\nu$, defined as
\begin{equation}
\nu(A) = \frac{\mu(A\cap D_\varepsilon)}{D_\varepsilon}.
\end{equation}
Now, if $H_\alpha(D_\varepsilon) =0$, then by a result of Boshernitzan cited in the introduction \eqref{eq:wstepbosh2}, we have 
\begin{equation}\liminf_{k\to+\infty} k^\beta d(\tau^k(x),x)=0 \mbox{ for $\nu$--a.e. $x \in D_\varepsilon$}.\label{eq:proofbosh}\end{equation}
Denote by $n_k(x)$ the time of $k$-th return of $x$ to $D_\varepsilon$. Then $\tau^k(x) = T^{n_k(y)}(x)$ and also \begin{equation} \lim_{k\to\infty} \frac{k}{n_k(x)} = \mu(D_\varepsilon) \end{equation} for $\nu$--a.e. $y$ because of the ergodic theorem.
Combining the two limits above we get
\begin{equation}
\liminf_{k\to+\infty} \mu(D_\varepsilon)^\beta \cdot n_k(x)^\beta d(T^{n_k(x)}(x),x)=0 \mbox{ for $\nu$--a.e. $x \in D_\varepsilon$}
\label{eq:proofbosh2}
\end{equation}
which is contradicts the definition of $D_\varepsilon$. Thus we know that $H_\alpha(D_\varepsilon)>0$. 

As the Hausdorff measure is positve, there must exists a measurable, non-empty subset $U\subset D_\varepsilon$ with $\operatorname{diam} U<\varepsilon$ satisfying $ (1-\varepsilon)(\operatorname{diam} U)^\alpha \leq H_\alpha(U)$. (If all subsets of $D_\varepsilon$ satisfy the opposite inequality, then this trivially violates the definition of the Hausdorff measure). Put $u=\mu(U)$ and $r=\operatorname{diam} U$. 
From the definition of density $g$ 
\begin{equation} \nonumber H_\alpha(U) = \int_X \1_{U} \,dH_\alpha = \int_X g \1_{U} \,d\mu \leq \esssup g \cdot \mu(U)= s\mu(U).\end{equation} 
Combining the two inequalities on $U$ gives 
\begin{equation} \label{uka}
 (1-\varepsilon) r^\alpha \leq s\cdot u.
\end{equation}
Since $T$ preserves $\mu$, we can show that 
\begin{equation} \label{tt13}
T^{-n}U \cap U \neq \emptyset \mbox{ for some } n \leq \frac{1}{u}.
\end{equation}
Indeed, if all those intersections were empty, then $U$, $T^{-1}(U)$ \ldots $T^{-[1/u]}(U)$ would be pairwise disjoint, and so \[\mu\bigg(\bigcup_{i=0}^{[1/u]}T^{-i}(U)\bigg) = \bigg(\Big[\frac{1}{u}\Big]+1\bigg)\cdot u  > 1,\]
a contradiction.

Take $n$ for which the intersection is non-empty and any $x\in T^{-n}U \cap U$. Then \begin{equation}\label{tt14} d(T^n(x),x) \leq \mbox{diam }U = r<\varepsilon,\end{equation}
so this $n$ satisfies the condition in the definition of $D(\varepsilon)$.
Using  (\ref{tt13}) and (\ref{tt14}), then (\ref{uka}), we get 
\begin{align*} \nonumber n^\beta d(T^n(x),x) &\leq \left(\frac{1}{u}\right)^\beta\cdot r  \leq \left(\frac{s}{1-\varepsilon}\right)^\beta\frac{1}{r}\cdot  r=\left(\frac{s}{1-\varepsilon}\right)^\beta, 
\end{align*}
which contradicts the the definition of $D_\varepsilon$ and ends the proof.
\end{proof} 

As the next step we will 'localize' the result above obtaining:
\begin{prop}\label{lokcor}
With the assumptions as above, we have for $\mu$ -- almost every $x\in X$
\begin{equation}
\liminf_{n\to \infty} \;n^{1/\alpha}d(T^n(x),x) \leq g(x)^{1/\alpha}.
\end{equation}
\end{prop}
\begin{rem} The density $g$ is defined only almost everywhere, so $g(x)$ really means
\begin{equation}\nonumber
g(x) = \lim_{r\to 0} (\esssup g|_{B(x,r)}).
\end{equation}\end{rem}
\begin{rem}
This result for $X=[0,1]$ ($\alpha=1$), has been proved in \cite{Choe}. That proof, however, works only in a 1-dimensional space.
\end{rem}

\begin{proof}[Proof of Proposition \ref{lokcor}] We will use basic ergodic properties as in the part of the previous proof. Fix $x$ and $r>0$ and consider $S(y)$ --- the first return function to the ball $B(x,r)$. $S$ preserves the conditional measure $\nu$, defined as
\begin{equation}
\nu(A) = \frac{\mu(A\cap B(x,r))}{\mu(B(x,r))}.
\end{equation}
The density of this new measure is related to the \emph{old} density: 
\begin{equation}\label{finwzornah} h = \frac{dH_\alpha}{d\nu} = g|_{B(x,r)} \cdot \mu(B(x,r)). \end{equation}
Using Proposition \ref{fingrandens} for a system $\left(B(x,r), \mathcal{F}|_{B(x,r)}, \nu, d|_{B(x,r)}, S\right)$ we get
\begin{equation} \label{finlimzsk}
\liminf_{k\to \infty} \;k^{1/\alpha}d(S^k(y),y) \leq (\esssup h)^{1/\alpha}.
\end{equation}
Denote by $n_k(y)$ the time of $k$-th return of $y$ to $B(x,r)$. Then $S^k(y) = T^{n_k(y)}(y)$ and from ergodic theorem \begin{equation} \lim_{k\to\infty} \frac{k}{n_k(y)} = \mu(B(x,r)) \mbox{\quad for $\mu$--a.e. $y$.}\end{equation}  Note that the closest returns (to itself) of a point $y\in B(x,r)$ for the original system have to occur within the sequence $n_k(y)$. Thus, the limit in (\ref{finlimzsk}) transforms to
\begin{equation} \label{finlimzsk2}
\liminf_{k\to \infty} \;\left(\frac{k}{n_k(y)}\right)^{1/\alpha} n_k(y)^{1/\alpha}\cdot d(T^{n_k(y)}(y),y)\geq \mu(B(x,r))^{1/\alpha}\cdot\liminf_{n\to \infty} \;n^{1/\alpha}d(T^n(y),y).
\end{equation}
It remains to compile (\ref{finwzornah}), (\ref{finlimzsk}) and (\ref{finlimzsk2}) obtaining
\begin{equation}
\mu(B(x,r))^{1/\alpha}\cdot\liminf_{n\to \infty} \;n^{1/\alpha}d(T^n(y),y) \leq (\esssup g|_{B(x,r)})^{1/\alpha}\cdot \mu(B(x,r))^{1/\alpha}.
\end{equation}
Letting $r \to 0$ we finish the proof.\end{proof}
We may finally finish the proof of the theorem.
\begin{proof}[Proof of Theorem \ref{thm:main}]
First, if $x\notin \mathrm{supp}\, \mu$, then as $g(x)$ is infinite and the bound is trivial. If $\mu \bot H_\alpha$, then $g(x)=0$ and the limit is zero by a result of Boshernitzan cited in the introduction \eqref{eq:wstepbosh2}. Finally, the remaining case is dealt with by Proposition~\ref{lokcor}.
\end{proof}

\end{document}